\documentclass{patmorin}
\usepackage[UKenglish]{babel}
\usepackage{tabularray}
\usepackage{inputenc}
\usepackage[T1]{fontenc}
\usepackage{float,verbatim}
\usepackage{graphicx,amsmath,amsthm,amsfonts,amssymb,mathtools}

\usepackage[capitalise, compress, nameinlink, noabbrev]{cleveref}
\crefname{lem}{Lemma}{Lemmas}
\crefname{thm}{Theorem}{Theorems}
\crefname{cor}{Corollary}{Corollaries}
\crefname{prop}{Proposition}{Propositions}
\newcommand{\defn}[1]{\textcolor{Maroon}{\emph{#1}}}
\newcommand{\mathdefn}[1]{\textcolor{Maroon}{#1}}       
\usepackage[longnamesfirst,numbers,sort&compress]{natbib}
 \makeatletter
 \def\NAT@spacechar{~}
 \makeatother
\makeatletter
\newcommand*\bigcdot{\mathpalette\bigcdot@{.5}}
\newcommand*\bigcdot@[2]{\mathbin{\vcenter{\hbox{\scalebox{#2}{$\m@th#1\bullet$}}}}}
\makeatother
\usepackage{todonotes}
\graphicspath{{figs/}}
\usepackage{paralist}
\newcommand{\boxprod}{\mathbin{\Box}}
\newcommand{\N}{\mathbb{N}}
\DeclarePairedDelimiter{\floor}{\lfloor}{\rfloor}
\DeclarePairedDelimiter{\ceil}{\lceil}{\rceil}

\renewcommand{\epsilon}{\varepsilon}
\renewcommand{\emptyset}{\varnothing}
\renewcommand{\ge}{\geqslant}
\renewcommand{\le}{\leqslant}
\renewcommand{\geq}{\geqslant}
\renewcommand{\leq}{\leqslant}
\DeclareMathOperator{\polylog}{polylog}

\DeclareMathOperator{\tw}{tw}
\DeclareMathOperator{\gm}{gm}

\newcommand{\GG}{\mathcal{G}}

 \renewcommand{\thefootnote}{\fnsymbol{footnote}}
\theoremstyle{plain}
\newtheorem{thm}{Theorem}
\newtheorem{lem}[thm]{Lemma}
\newtheorem{cor}[thm]{Corollary}

\newtheorem{obs}[thm]{Observation}

\crefname{obs}{Observation}{Observations}
\newtheorem*{lem*}{Lemma}
\theoremstyle{definition}

\newtheorem*{conj*}{Conjecture}
\let\oldsubset\subset
\newcommand{\subsetsim}{\mathrel{\ooalign{\raise0.175ex\hbox{$\oldsubset$}\cr\hidewidth\raise-0.9ex\hbox{\scalebox{0.9}{$\sim$}}\hidewidth\cr}}}
\begin{document}
\title{\MakeUppercase{Grid Minors and Products}}

\author{
Vida Dujmovi{\'c}\,\footnotemark[3]\qquad
Pat Morin\,\footnotemark[4]\qquad
David~R.~Wood\,\footnotemark[2]\qquad
David~Worley\footnotemark[5]
}

\maketitle

\begin{abstract}
  Motivated by recent developments regarding the product structure of planar graphs, we study relationships between treewidth, grid minors, and graph products.  We show that the Cartesian product of any two connected $n$-vertex graphs contains an $\Omega(\sqrt{n})\times\Omega(\sqrt{n})$ grid minor.  This result is tight: The lexicographic product (which includes the Cartesian product as a subgraph) of a star and any $n$-vertex tree has no $\omega(\sqrt{n})\times\omega(\sqrt{n})$ grid minor.
\end{abstract}

\footnotetext[3]{School of Computer Science and Electrical Engineering, University of Ottawa, Ottawa, Canada (\texttt{vida.dujmovic@uottawa.ca}). Research supported by NSERC.}

\footnotetext[4]{School of Computer Science, Carleton University, Ottawa, Canada (\texttt{morin@scs.carleton.ca}). Research supported by NSERC.}

\footnotetext[2]{School of Mathematics, Monash University, Melbourne, Australia (\texttt{david.wood@monash.edu}). Research supported by the Australian Research Council. }

\footnotetext[5]{School of Computer Science and Electrical Engineering, University of Ottawa, Ottawa, Canada (\texttt{dworl020@uottawa.ca}). Research supported by NSERC.}

 \renewcommand{\thefootnote}{\arabic{footnote}}

\section{Introduction}
\label{Intro}

Treewidth\footnote{A \defn{tree-decomposition} of a graph $G$ is a collection $(B_x :x\in V(T))$ of subsets of $V(G)$ (called \defn{bags}) indexed by the vertices of a tree $T$, such that (a) for every edge $uv\in E(G)$, some bag $B_x$ contains both $u$ and $v$, and (b) for every vertex $v\in V(G)$, the set $\{x\in V(T):v\in B_x\}$ induces a non-empty (connected) subtree of $T$. The \defn{width} of $(B_x:x\in V(T))$ is $\max\{|B_x| \colon x\in V(T)\}-1$. 
%A \defn{path-decomposition} is a tree-decomposition in which the underlying tree is a path, simply denoted by the corresponding sequence of bags $(B_1,\dots,B_n)$. 
    The \defn{treewidth} of a graph $G$, denoted by \defn{$\tw(G)$}, is the minimum width of a tree-decomposition of $G$.} is a ubiquitous parameter in structural graph theory measuring how close a graph is to a tree. It was first introduced as \textit{dimension} by \citet[pp.~37--38]{BB1972} in 1972, then rediscovered by \citet{Halin76} in 1976. The parameter was popularized when it was once again rediscovered by Robertson and Seymour~\cite{ROBERTSON198449} in 1984 and has since been at the forefront of structural graph theory research. Graphs of bounded treewidth are of particular interest due to the wide implications of their tree-like structure. For example, Courcelle's Theorem~\cite{Courcelle90} implies that many NP-Complete problems can be solved in linear time on graphs of bounded treewidth. 

Another key way to study the structure of graphs is through analysis of their minors.\footnote{A graph $G_1$ is a \defn{minor} of a graph $G_2$ if a graph isomorphic to $G_1$ can be obtained from a subgraph of $G_2$ through a sequence of edge contractions.} Finding a highly structured minor of a graph $G$ allows us to use properties of the minor to study $G$ itself. 
Our focus is on grid minors. 
%In particular, we are interested in studying the largest $k$ such that the $k\times k$ grid graph $\boxplus_k$ is a minor of $G$, where 
Let \defn{$\boxplus_k$} be the \defn{$k\times k$ grid}, which is the graph with vertex-set $\{1,\dots,k\}^2$ and edge-set $\{(x,y)(x',y'):|x-x'|+|y-y'|=1,\,x,y,x',y'\in \{1,\dots,k\}\}$.  
%\pat{This definition looks fishy: It says that $(5,1)$ is adjacent to $(1,4)$.  Replace the condition with $|x-x'|+|y-y'|=1$.}\worley{done.} \david{good spot}
%%%%%%%%
%\david{Cartesian product is not yet defined, and later we say $V(\boxplus_k)=\{1,\dots,k\}^2$. So I suggest we write here ``where $\boxplus_k$ is the graph with vertex-set $\{1,\dots,k\}^2$ and edge-set $\{(x,y)(x',y'):|x+y-x'-y'|=1,\,x,y,x',y'\in \{1,\dots,k\}\}$.'' This could go into a footnote if you prefer. If this edge-set definition is too brief, then write $\{(x,y)(x+1,y): x\in\{1,\dots,k-1\},y\in\{1,\dots,k\}\} \cup \{(x,y)(x,y+1): x\in\{1,\dots,k\},y\in\{1,\dots,k-1\}\}$.} \worley{I think the brief version is good, I have added it for now and will also discuss with Pat}. \david{I reworded slightly.} 
%%%%%%%%
For a graph $G$, let \defn{$\gm(G)$} be the maximum integer $k$ such that $\boxplus_k$ is a minor of $G$. 
It is folklore that $\tw(\boxplus_k)=k$; see~\citep{HW17} for a proof. 
This, combined with the fact that treewidth is minor-monotone, implies that for every graph $G$, 
\begin{equation}\label{tw_gte_gm}
    \tw(G) \geq \gm(G).
\end{equation}

%A similar structurally impactful parameter is the  Hadwiger Number. The Hadwiger Number of a graph $G$, denoted $h(G)$ is the largest integer $n$ is such that $G$ contains a $K_n$ minor. We refer interested readers to \cite{todo} \worley{Find a survey for this} for a survey on this area, as our work will focus on the aformentioned grid minors. 

A foundational result in this area is the \defn{Grid Minor Theorem} (also known as the \defn{Excluded Grid Theorem}) of Robertson and Seymour~\cite{RS-V}, which says there exists a function $f$ such that for every positive integer $k$, $f(k)$ is the minimum integer such that every graph with treewidth at least $f(k)$ contains a $k \times k$ grid minor.  Thus grid graphs are canonical examples of graphs with large treewidth. 

This result had widespread impact on structural graph theory research and led to further investigation into the best possible bounds for the function $f$. \citet{RS-V} proved the existence of $f(k)$, which they, along with Thomas, later showed to be in $2^{O(k^5)}$~\cite{RST94}. \citet{DJGT-JCTB99} showed that if $G$ has treewidth $\Omega(k^{4m^2(k+2)})$ where $k$ and $m$ are integers, then $G$ contains either $K_m$ or the $k \times k$ grid as a minor. \citet{LeafSeymour15} improved the upper bound to $f\in 2^{O(k \log k)}$. The first polynomial upper bound, stating that $f \in O(k^{98}\log k)$, was found by Chekuri and Chuzhoy~\cite{CC16}. Chuzhoy continued to work towards lowering this exponent, with the current state-of-the-art result by \citet{CT21} showing that $f \in O(k^9\log^{O(1)}k)$. A lower bound of $f\in \Omega(k^2\log k)$ was shown by Robertson~et~al.~\cite{RST94}, and \citet{DHK-Algo09} conjectured $f\in \Theta(k^3)$. 

For particular classes of graphs, much stronger Grid Minor Theorems are known. Say a class $\GG$ has the \defn{linear grid minor property} if, for some constant $c$,  every graph in $\mathcal{G}$ with treewidth at least $ck $ contains $\boxplus_k$ as a minor. For example, 
\citet{RST94} showed that every planar graph with treewidth at least $6k$ contains $\boxplus_k$ as a minor. Thus the class of planar graphs has the linear grid minor property. More generally, \citet{DemHaj-EuJC07} proved that every proper minor-closed class has the linear grid minor property. The proof used the Graph Minor Structure Theorem, which in turn depends on the Grid Minor Theorem. \citet{KK20} gave an alternative self-contained proof. In particular, they showed that for any graph $H$ there exists $c\leq |V(H)|^{O(|E(H)|)}$ such that 
every $H$-minor-free graph with treewidth at least $ck$ contains $\boxplus_k$ as minor. 
The linear grid minor property has been used to devise efficient polynomial time approximation schemes for many NP-hard problems on planar graphs and related graph families~\cite{DHK-Algo09,DFHT-JACM05,DH-Algo04, Eppstein-Algo00, FFLS18}. Note that the $\Omega(k^2\log k)$ lower bound mentioned above shows that general graphs do not have the linear grid minor property.

In this paper, we study grid minors in graph products\footnote{See \citep{EH20,HW23,WoodJGT13,EFM22,KOY14,CK06} for related work on the treewidth of graph products, and see \citep{BM98,Wood-NYJM11} for related work on complete graph minors in graph products.}.  This is motivated both by the fact that the grid graph itself is isomorphic to the Cartesian product of two paths, and also by recent developments in Graph Product Structure Theory. This area of research studies complex graph classes by modelling them as a product of simpler graphs and investigating the properties of these highly structured supergraphs. 
Before discussing Graph Product Structure Theory further, we first define the three types of graph products that we consider, each illustrated in \cref{Products}. Let $G_1$ and $G_2$ be graphs. 
%\worley{Updated all product definitions to unify indexing so that a vertex $(u_1, v_1)$ belongs to $G_1$ according to reviewer A request} \david{Are we consistent with this throughout the paper?} \worley{All indexing looks to be either consistent with this, or contextually clear to me} \david{okay} 
The \defn{Cartesian product} of  $G_1$ and $G_2$, denoted $G_1 \boxprod G_2$, is the graph with vertex set $V(G_1) \times V(G_2)$ where two distinct vertices $(u_1, u_2)$ and $(v_1, v_2)$ are adjacent if and only if: 
\begin{compactitem}
    \item $u_1 = v_1$ and $u_2v_2 \in E(G_2)$, or
    \item $u_2 = v_2$ and $u_1v_1 \in E(G_1)$.
\end{compactitem}
The \defn{strong product} of  $G_1$ and $G_2$, denoted $G_1 \boxtimes G_2$, is the graph with vertex set $V(G_1) \times V(G_2)$ where two distinct vertices $(u_1, u_2)$ and $(v_1, v_2)$ are adjacent if and only if: 
\begin{compactitem}
    \item $u_1 = v_1$ and $u_2v_2 \in E(G_2)$, 
    \item $u_2 = v_2$ and $u_1v_1 \in E(G_1)$, or
    \item $u_1v_1 \in E(G_1)$ and $u_2v_2 \in E(G_2)$.
\end{compactitem}
%\worley{One reviewer suggests using a thicker cdot for lexicographic product. Hows this look?} \david{looks fine}

The \defn{lexicographic product} of  $G_1$ and $G_2$, denoted $G_1 \bigcdot G_2$, is the graph with vertex set $V(G_1) \times V(G_2)$ where two distinct vertices $(u_1, u_2)$ and $(v_1, v_2)$ are adjacent if and only if:
\begin{compactitem}
    \item $u_1v_1 \in E(G_1)$, or
    \item $u_1 = v_1$ and $u_2v_2 \in E(G_2)$.
\end{compactitem}
It follows from the above definitions that 
$$G_1\boxprod G_2 \subseteq G_1\boxtimes G_2 \subseteq G_1 \bigcdot G_2.$$

% \begin{figure}[ht]
% \begin{center}
%   \includegraphics[scale=0.7]{figs/pathProducts-3.ipe.png}
% \end{center}
% \caption{From left to right: $P_4\boxprod P_4$, $P_4\boxtimes P_4$, and $P_4\cdot P_4$ .}
% \label{Products}
% \end{figure}

\begin{figure}[ht]
% \begin{center}
\begin{tblr}{colspec={X[c,m]X[c,m]X[c,m]X[c,m]X[c,m]X[c,m]}}
 \raisebox{.3\height}{\includegraphics[page=1]{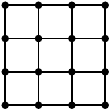}} &
 \raisebox{.3\height}{\includegraphics[page=2]{figs/P_4xP_4}} &
 \raisebox{.3\height}{\includegraphics[page=3]{figs/P_4xP_4}} &
 \includegraphics[page=1]{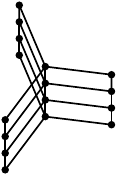} &
 \includegraphics[page=2]{figs/S_3xP_4} &
 \includegraphics[page=3]{figs/S_3xP_4} \\
 $P_4\boxprod P_4$ & $P_4\boxtimes P_4$ & $P_4\bigcdot P_4$ &
 $S_3\boxprod P_4$ & $S_3\boxtimes P_4$ & $S_3\bigcdot P_4$ 
\end{tblr}
% \end{center}
\caption{The products of two paths and of a star and a path.}
\label{Products}
\end{figure}

The starting point for recent developments in Graph Product Structure Theory is the \defn{Planar Graph Product Structure Theorem} of \citet{DJMMUW20}, which states that every planar graph $G$ is isomorphic to a subgraph of the strong
product of two very simple graphs, a graph $H$ of treewidth\footnote{This treewidth bound was improved to 6 by~\citet{UWY22}.} at most 8 and a path $P$, written as $G\subsetsim H\boxtimes P$.  Although $H\boxtimes P$ is a supergraph of the original graph $G$, it often shares or inherits properties of $G$, and its rigid structure makes it easier to work with. For example, 
 planar graphs with diameter $k$ have treewidth $O(k)$, and similarly any subgraph of $H\boxtimes P$  of diameter $k$ has treewidth $O(k)$.
% \referee{Page 4 top, "induced subgraph of $H\boxtimes P$ of diameter $k$ has treewidth $O(k)$", but this subgraph has no distance-preserving relation to a planar  subgraph of $H\boxtimes P$, and so this remark makes little sense here.} \david{I think Pat wrote this. I expect Pat meant that planar graphs with diameter $k$ also have treewidth $O(k)$. We should make this explicit. I suggest we write  ``For example,  planar graphs with diameter $k$ have treewidth $O(k)$, and similarly any subgraph of $H\boxtimes P$  of diameter $k$ has treewidth $O(k)$.''}\worley{Sounds good, I've included the change}. 
For a more global example, any $n$-vertex subgraph of $H\boxtimes P$ has a balanced separator of size $O(\sqrt{n})$ (see \cite[Lemma~6]{DJMMUW20} and \cite[Lemma~10]{DMW17}).  The proofs of both of these facts are considerably simpler than the same result for planar graphs. 

Product structure theorems have also been developed for other classes of graphs such as surface embeddable graphs~\citep{DJMMUW20,DHHW22}, graphs excluding an apex minor~\citep{DJMMUW20,ISW24,DHHJLMMRW}, graphs excluding any fixed minor~\citep{DJMMUW20,DEMWW22}, and various non-minor-closed classes~\citep{DMW23,HW24,DHSW24,BDHK24}. These product structure theorems have been key in resolving many long-standing open problems on queue layouts~\citep{DJMMUW20}, non-repetitive colourings~\citep{DEJWW20}, centred colourings~\citep{DFMS21,DHHJLMMRW}, adjacency labelling~\cite{EJM23,DEGJMM21}, twin-width~\cite{BKW,KPS24,JP22}, vertex ranking~\citep{BDJM}, and box dimension~\citep{DGLTU22}. This wide range of applications motivates the need for a deeper understanding of structural properties of graph products. 

\subsection{Our Results}
% \worley{Do we want to separate out footnote 5 as a claim?} \david{I would leave it as is}

It is known\footnote{Let $G_1$ and $G_2$ be connected graphs each with at least $n$ vertices. 
For $i\in\{1,2\}$, let $v_i$ be a leaf of a spanning tree of $G_i$, and let $G'_i:=G_i-v_i$, which is connected. For each $x\in V(G'_1)$, let $B_x$ be the subgraph of $G_1\boxprod G_2$ induced by $\{x\}\times V(G'_2)$. 
For each $y\in V(G'_2)$, let $B_y$ be the subgraph of $G_1\boxprod G_2$ induced by $V(G'_1)\times \{y\}$. 
Let $B_1$ be the subgraph of $G_1\boxprod G_2$ induced by $\{v_1\}\times V(G_2)$. 
Let $B_2$ be the subgraph of $G_1\boxprod G_2$ induced by $V(G'_1)\times \{v_2\}$. 
Let $\mathcal{B}:=\{ B_x\cup B_y: x\in V(G'_1),y\in V(G'_2)\}\cup\{B_1,B_2\}$. Then it is easily seen that $\mathcal{B}$ is a bramble in $G_1\boxprod G_2$ of order at least $n+1$. By the Treewidth Duality Theorem~\citep{ST93}, $\tw(G_1\boxprod G_2)\geq n$. This result was extended by \citet{WoodJGT13} who showed that 
for any two $k$-connected graphs $G$ and $H$ each with at least $n$ vertices, $\tw(G\boxprod H) \geq k(n-2k+2)-1$. } that for all $n$-vertex connected graphs $G_1$ and $G_2$, 
\begin{equation}
    \label{LowerBoundProduct}
    \tw(G_1 \boxprod G_2) \ge n.
\end{equation}
It thus makes sense for a grid minor theorem for graph products to be in terms of $n$. 
We show that this is in fact the case
by proving the following results:

\begin{enumerate}
   \item  For any two $n$-vertex connected graphs $G_1$ and $G_2$, 
   $$\gm(G_1\bigcdot G_2) \geq \gm(G_1\boxtimes G_2) \geq \gm(G_1\boxprod G_2) \in \Omega(\sqrt{n})\;\;\text{(see \cref{lower_bound})}.$$
      
   \item There exists two $n$-vertex connected graphs $G_1$ and $G_2$ (a star and any tree) such that   $$\gm(G_1\boxprod G_2) \leq \gm(G_1\boxtimes G_2) \leq \gm(G_1\bigcdot G_2) \in O(\sqrt{n}) \;\; \text{(see \cref{StarTreeUpperBound})}.$$
   
\end{enumerate}

The previous best bound for the product of two $n$-vertex connected graphs comes from combining \eqref{LowerBoundProduct} with the state-of-the-art Grid Minor Theorem  of Chuzhoy and Tan~\cite{CT21}, giving $\gm(G_1 \boxprod G_2) \in \Omega(n^{1/9}/\polylog(n))$. The first result above gives an excluded grid theorem for graph products that is stronger than what is possible for general graphs and much stronger than what can be proven for general graphs.

The second result shows that the first result is tight for the Cartesian, strong, and lexicographic product of two trees. A consequence of the second result and \eqref{LowerBoundProduct} is that there exists two trees whose Cartesian product has treewidth at least $n$ but whose largest grid minor has size $O(\sqrt{n})\times O(\sqrt{n})$. Thus, even these simple products do not have the linear (or even subquadratic) grid minor property. 

The remainder of this paper is organized as follows:  In \cref{A} we discuss background material. In \cref{B} we prove the above lower bound. In \cref{C} we prove the above upper bound. \cref{D} proves some exact bounds for the grid minor number of the products of stars and trees. Finally, \cref{E} concludes with directions for future work.

\section{Preliminaries}\label{A}

For any standard graph-theoretic terminology and notation not defined here, we use the same conventions used in the textbook by~\citet{D10}. In this paper, every graph $G$ is undirected and simple with vertex set $V(G)$ and edge set $E(G)$. The \defn{order} of $G$ is denoted $|G|:=|V(G)|$.  For two graphs $G_1$ and $G_2$ we use the notation $G_1\preceq G_2$ to indicate that $G_1$ is a minor of $G_2$.  We make use of the fact that the $\preceq$ relation is transitive. The following observation follows immediately from definitions.

\begin{obs}\label{minor_product}
  Let $G_1$, $G_2$, and $H$ be graphs.  If $G_1\preceq G_2$, then $G_1\boxprod H\preceq G_2\boxprod H$.
\end{obs}

A \defn{model} of a graph $H$ in a graph $G$ is a set $\mathcal{M}:=\{B_x\subseteq V(G): x\in V(H)\}$ of subsets of $V(G)$, called \defn{branch sets}, indexed by the vertices of $H$ and such that:
\begin{compactenum}[(i)]
  \item for each distinct pair $x,y\in V(H)$, $B_x\cap B_y=\emptyset$;
  \item for each $x\in V(H)$, $G[B_x]$ is connected and
  \item for each $xy\in E(H)$ there exists an edge $vw\in E(G)$ with $v\in B_x$ and $w\in B_y$.
\end{compactenum}
It follows from definitions that $H\preceq G$ if and only if there exists a model of $H$ in $G$.

For each $n\in \N$, let $S_n$ denote the \defn{$n$-star}; the rooted tree with $n$ leaves, each of which is adjacent to the root.  For each $\ell,p\in\N$, let $S_{\ell,p}$ denote the star with $\ell$ leaves whose edges have been subdivided $p-1$ times.  More formally, $V(S_{\ell,p}):=\{v_0\}\cup\{v_{i,j}:(i,j)\in[\ell]\times[p]\}$ and $E(S_{\ell,p}):=\{v_0v_{i,1}:i\in[\ell]\}\cup \{v_{i,j}v_{i,j+1}:(i,j)\in[\ell]\times[p-1]\}$, where $[\ell]=\{1,\ldots,\ell\}$.  We call $S_{\ell,p}$ a \defn{subdivided star}.  Subdivided stars generalize both stars and paths: The $n$-vertex path $P_n$ is isomorphic to $S_{1,n-1}$ and the $n$-leaf star $S_n$ is isomorphic to $S_{n,1}$. 

\begin{lem}\label{anything_times_star}
  For any positive integer $n$ and any $n$-vertex connected graph $G$, $K_{n} \preceq G\boxprod S_n$.
\end{lem}

Note that this lemma is implied by \citep[Lemma~5.1]{Wood-NYJM11}; we include the proof here for the sake of completeness.
\begin{proof}
  Let $y_0$ denote the root of $S_n$, let $y_1,\ldots,y_n$ denote the leaves of $S_n$.  Let $V(K_{n})=\{1,\dots,n\}$ and let $v_1,\ldots,v_n$ denote the vertices of $G$.  
  We now construct a model $\mathcal{M}:=\{B_x:x\in V(K_n)\}$ of $K_n$ in $G\boxprod S_n$.  For each $i\in V(K_n)$, define the branch set
  \[
     B_i:=\{(v,y_i):v\in V(G)\} \cup \{ (v_{i},y_0) \} \enspace .
  \]
  We now show that $\mathcal{M}$ is a model of $K_n$ in $G\boxprod S_n$.
  The induced graph $(G\boxprod S_n)[B_i]$ is connected because $(G\boxprod S_n)[\{(v,y_i):v\in V(G)\}]$ 
  is isomorphic to $G$ and $(v_{i},y_0)$ is adjacent to $(v_{i},y_i)$.
  For any $1\le i< j\le n$, $B_i$ and $B_j$ are disjoint because $y_i\neq y_j$ and $v_i\neq v_j$.  Furthermore, the vertex $(v_{i},y_0)\in B_i$ is adjacent to $(v_i,y_j)\in B_j$.  Therefore, there is an edge in $G\boxprod S_n$ with one endpoint in $B_i$ and one endpoint of $B_j$ for each $1\le i < j\le n$.
\end{proof}

Note that, for any tree $T$ with $n$ leaves and at least one non-leaf vertex, $S_n\preceq T$.  In this case, \cref{anything_times_star,minor_product} imply that $K_n\preceq G\boxprod T$.

Finally, we mention the following upper bound on the treewidth of $G_1\bigcdot G_2$. This result is stated in \citep{HW23} for $G_1\boxtimes G_2$ and is implicit in earlier works; we include the easy proof for completeness.

\begin{lem}
\label{EasyUpperBound}
For any graphs $G_1$ and $G_2$,
$$\tw(G_1\boxprod G_2)\leq\tw(G_1\boxtimes G_2)\leq \tw(G_1\bigcdot G_2)\leq (\tw(G_1)+1)|V(G_2)|-1.$$
\end{lem}

\begin{proof}
The first two inequalities hold since $G_1\boxprod G_2 \subseteq G_1\boxtimes G_2 \subseteq G_1\bigcdot G_2$. For the final inequality, start with a tree-decomposition $(B_x:x\in V(T))$ of $G_1$ with bags of size at most $\tw(G_1)+1$. For each $x\in V(T)$ let $B'_x:=\{(v,w):v\in B_x,w\in V(G_2)\}$. Observe that $(B'_x:x\in V(T))$ is a tree-decomposition of $G_1\bigcdot G_2$, and $|B'_x|\leq |B_x|\,|V(G_2)| \leq (\tw(G_1)+1)|V(G_2)|$ for each $x\in V(T)$. The result follows. 
\end{proof}

\cref{LowerBoundProduct,EasyUpperBound} imply that for any trees $T_1$ and $T_2$,
\begin{equation}
\label{TreewidthProductTrees}
\min\{|V(T_1)|,|V(T_2)|\} \leq
\tw(T_1\boxprod T_2)\leq
\tw(T_1\boxtimes T_2)\leq 
\tw(T_1\bigcdot T_2)\leq 
2\min\{|V(T_1)|,|V(T_2)|\}-1.
\end{equation}
Thus the treewidth of the Cartesian, strong or lexicographic product of two trees is determined (up to a factor of 2) by the number of vertices in the two trees. The remainder of the paper shows that determining the largest grid minor in such a product is more nuanced. 
% \worley{I like the addition} \david{done}

\section{The Lower Bound}\label{B}

% Our lower bound, which shows that, for any two $n$-vertex connected graphs $G_1$ and $G_2$, $\gm(G_1\boxprod G_2)\in\Omega(\sqrt{n})$ is obtained as follows:  We first show that each of $G_1$ and $G_2$ contains a `large' subdivided star.  Then we show that the product of any two `large' subdivided stars $S_{\ell_1, p_1}$ and $S_{\ell_2, p_2}$ contains a large grid-minor.  The quotes on `large' here are due to the fact that we measure the size of a subdivided star in a non-standard way, as in the following lemma.

\subsection{Connected Graphs Contain Large Subdivided Stars}

We first state some terminology that will be relevant in the following results. The \defn{length} of a path $v_0,\ldots,v_r$ is the number, $r$, of edges in the path. The \defn{depth} of a vertex $v$ in a rooted tree $T$ is the length of the path, in $T$, from $v$ to the root of $T$. A path $P$ in a rooted tree $T$ is \defn{vertical} if, for each $i\in\N$, $V(P)$ contains at most one vertex of depth $i$. The vertex of minimum depth in a vertical path is its \defn{upper endpoint}, and the vertex of maximum depth in a vertical path is its \defn{lower endpoint}. A vertex $v$ is a \defn{$T$-ancestor} of a vertex $w$ if the vertical path from $w$ to the root of $T$ contains $v$.  Two vertices of $T$ are \defn{unrelated} if neither is a $T$-ancestor of the other, otherwise they are \defn{related}.  A pair of paths $P_1$ and $P_2$ in $T$ is \defn{completely unrelated} if $v$ and $w$ are unrelated, for each $v\in V(P_1)$ and each $w\in V(P_2)$.  We say that $P_1$ and $P_2$ are \defn{completely related} if $v$ and $w$ are related, for each $v\in V(P_1)$ and each $w\in V(P_2)$. The \defn{height} $h_T(v)$ of a vertex $v$ in $T$ is the maximum order of a vertical path in $T$ whose upper endpoint is $v$.  For each $i\in\N$, let $\mathdefn{H_i(T)}:=\{v\in V(T):h_T(v)=i\}$ and $\mathdefn{n_i(T)}:=|H_i(T)|$.   We have the following observation:

\begin{obs}\label{same_height_unrelated}
  For any rooted tree $T$ and any $i\in\N$, $T$ contains a set of $n_i(T)$ pairwise completely unrelated vertical paths, each of order $i$.  As a consequence, $S_{n_i(T),i}\preceq T$ for each $i\in\N$. 
\end{obs}

\begin{proof}
  Let $v_1,\ldots,v_{n_i(T)}:= H_i(T)$ and observe that $v_1,\ldots,v_{n_i(T)}$ are pairwise unrelated.  For each $j\in\{1,\ldots,n_i(T)\}$, let $P_j$ be a path of order $i$ that has $v_j$ as an upper endpoint. (Such a path exists by the definition of $H_i(T)$.)  Observe that, for distinct $j$ and $k$, $P_j$ and $P_k$ are vertex-disjoint, and completely unrelated since $v_j$ and $v_k$ are unrelated. By contracting each edge that has both endpoints of depth less than $i$ into a single vertex $x$ and removing all vertices not in $\{x\}\cup\bigcup_{j\in\{1,\ldots,n_i(T)\}} V(P_j)$ we obtain $S_{n_i(T),i}$. Thus $S_{n_i(T),i}\preceq T$. 
\end{proof}

We will show that the product $G_1\boxprod G_2$ of two connected $n$-vertex graphs $G_1$ and $G_2$ contains an $\Omega(\sqrt{n})\times\Omega(\sqrt{n})$ grid minor by studying the product $T_1\boxprod T_2$ of two spanning trees of $G_1$ and $G_2$, respectively. \Cref{anything_times_star} allow us to dispense with the case when $n_i(T_b)\in\Omega(n)$ for some $i$ and some $b\in\{1,2\}$ since, if $n_i(T_b)\in\Omega(n)$, then $T_b$ contains a $S_{\Omega(n)}$-minor, so \cref{anything_times_star} implies $K_{\Omega(n)}\preceq T_1\boxprod T_2$, so $\boxplus_{\Omega(\sqrt{n})}\preceq T_1\boxprod T_2$.
% \david{I find the previous sentence strange, since 
% \cref{anything_times_star} (without \cref{same_height_unrelated}) implies the following stronger statement: if some $T_i$ has $p\in \Omega(n)$ non-root leaves, then $T_i$ has a star with $p$ leaves as a minor, so $K_p$ is a minor of $G_1\boxprod G_2$ by 
% \cref{anything_times_star}, implying $\boxplus_{\sqrt{p}}$ is a minor of $G_1\boxprod G_2$, implying $\gm(G_1\boxprod G_2)\in \Omega(\sqrt{n})$, and we are done.}. 
The following lemma will be helpful when this is not possible.  (In several places, including the following lemma, we make use of Euler's solution~\citep{EulerBasel} to the Basel Problem: $\sum_{i=1}^\infty 1/i^2 = \pi^2/6$.)

\begin{lem}\label{disjoint_p_paths}
  Let $T$ be a rooted tree with $n\ge 1$ vertices, and let $p\geq 1$ be an integer such that $n_i(T) \leq \tfrac{3}{2}n/(\pi i)^2$ for each $i\in\{1,\ldots,p-1\}$.  Then $T$ contains pairwise vertex-disjoint vertical paths $P_1,\ldots,P_{\ceil{n/4p}}$, each of order $p$ such that, for each $i\neq j$, $P_i$ and $P_j$ are either completely unrelated or completely related.
\end{lem}

\begin{proof}
  Let $T':=T-(\bigcup_{i=1}^{p-1} H_i(T))$ be the subtree of $T$ induced by vertices of height at least $p$.  Then,
  \[
    |T'|\ge |T| - \sum_{i=1}^{p-1} n_i(T)
    \ge n - \tfrac{3n}{2\pi^2}\sum_{i=1}^{p-1} \tfrac{1}{i^2}
    \ge  n - \tfrac{n}{4} = \tfrac{3n}{4} \enspace .
  \]
 Let $L$ be the set of non-root leaves of $T'$.  Each vertex in $L$ is the upper endpoint of a vertical path in $T$ of order $p$, as illustrated in \cref{disjoint_p_paths_figs}.  Therefore, if $|L|\ge \tfrac{n}{4p}$ then we are done, so assume that $|L|<\tfrac{n}{4p}$.  

  % \begin{figure}
  %   \begin{center}
  %     \includegraphics{figs/disjoint-paths}
  %   \end{center}
  %   \caption{Finding paths in the tree $T'$ induced by vertices of height at least $p=4$. Note that the figures are drawn so that for each edge, the left vertex is the upper endpoint.\worley{TODO: Update graphic w.r.t. reviewer A comment on middle branch leaf node of height one being kept in $T'$, and update to rotate the figure so it is vertical }} 
  %   \label{disjoint_p_paths_figs}
  % \end{figure}

%\worley{Replaced graphic with vertical version, with additional leaf node on middle branch now greyed out and excluded from $T'$} \david{good}

  \begin{figure}[h]
    \begin{center}
      \includegraphics[width=0.3\linewidth]{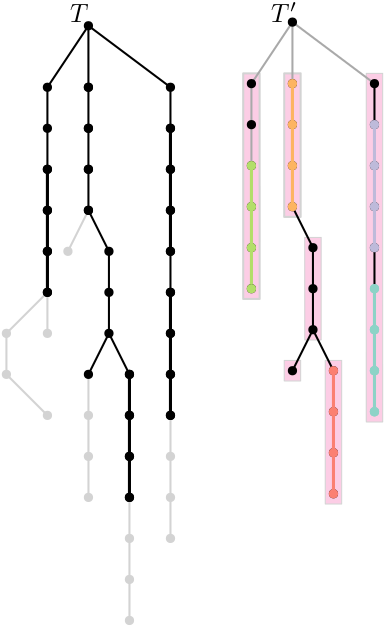}
    \end{center}
    \caption{Finding paths in the tree $T'$ induced by vertices of height at least $p=4$.} 
    \label{disjoint_p_paths_figs}
\end{figure}

  Let $S$ be the set of vertices of $T'$ that have two or more children in $T'$.  In any rooted tree, the number of non-root leaves is greater than the number of non-leaf vertices with at least two children\footnote{Let $n_i$ be the number of vertices with $i$ children in a rooted tree $T$. Thus $\sum_{i\geq 0} i n_i=|E(T)|<|V(T)|=\sum_{i\geq 0}n_i$. Hence, the number of non-root leaves 
  $n_0 > \sum_{i\geq 1} (i-1) n_i \geq \sum_{i\geq 2} n_i$, as claimed.}  Therefore, $|S|<|L|\le \tfrac{n}{4p}$.

  For each $v\in S\cup L$, let $P_v$ be the vertical path of maximum length whose lower endpoint is $v$ and that does not contain any vertex in $(S\cup L)\setminus \{v\}$.  Then $\mathcal{P}:=\{P_v:v\in S\cup L\}$ is a partition of $V(T')$ into at most $r:=|S|+|L|\le \tfrac{n}{2p}$ parts, each of which induces a vertical path in $T'$.

  Let $\{P_1,\ldots,P_r\}:=\mathcal{P}$ and, for each $i\in\{1,\ldots,r\}$, let $P'_i$ be a subpath of $P_i$ of order $p\floor{|P_i|/p}$. (So $P'_i$ has order rounded down to a multiple of $p$.)  Then 
  $$\sum_{i=1}^r |P_i'|\ge \sum_{i=1}^r (|P_i|-(p-1)) = |T'| - (p-1)r\ge \tfrac{3n}{4}-\tfrac{n}{2}=\tfrac{n}{4}.$$ 
  For each $i\in\{1,\ldots,r\}$, $P_i'$ can be partitioned into exactly $|P_i'|/p$ vertex-disjoint paths, each of order $p$.  The set $\mathcal{P}'$ of paths obtained this way has size $\ell := \sum_{i=1}^r |P_i'|/p \ge \tfrac{n}{4p}$.  Therefore $T$ contains $\ell$ pairwise vertex-disjoint paths, each of order $p$, where $\ell\ge \tfrac{n}{4p}$.  Except for its lower endpoint, each vertex of a path in $\mathcal{P}'$ has exactly one child in $T$.  This ensures that each path in $\mathcal{P}'$ is either completely related or completely unrelated to any other path in $\mathcal{P}'$.
\end{proof}

\subsection{The Product of Two Special Trees}

% We now complete the second part of our proof, by showing that the two trees guaranteed by \cref{subdivided_star_minor} and

\begin{lem}\label{star_times_star}
  Let $s,p\ge 1$ be integers, let $\ell:=5s^2$, and let $T$ be a rooted tree that contains $s^2$ pairwise-disjoint vertical paths, each of order $6p$ such that any pair of these paths is either completely related or completely unrelated.  Then $\gm(T\boxprod S_{\ell,2p})\ge sp$.
\end{lem}

\begin{proof}
  Recall that $S_{\ell,2p}$ has vertex set $\{v_0\}\cup\{v_{i,j}:(i,j)\in \{1,\ldots,\ell\}\times\{1,\ldots,2p\}\}$. For each $i\in\{1,\ldots,\ell\}$, let $A_i=S_{\ell,2p}[\{v_{i,1},\ldots,v_{i,2p}\}]$ denote the $i$th \defn{arm} of $S_{\ell,2p}$, which is a path of order $2p$.  Let $P_1,\ldots,P_{s^2}$ be pairwise vertex-disjoint paths in $T$, each of order $6p$, each pair of which is either completely related or completely unrelated.  For each $i\in\{1,\ldots,s^2 \}$, let $P_i:= p_{i,1},\ldots,p_{i,6p}$ where $p_{i,1}$ is the upper endpoint of $P_i$. 
  Let $T_0:=T\boxprod \{v_0\}$.  
  For each $i\in\{1,\ldots,\ell\}$ let $T_i:=T\boxprod A_i$, for each $j\in \{1,\ldots,p\}$ let $T_{i,j}:=T\boxprod \{v_{i,j}\}$, for each $k\in\{1,\ldots,s^2\}$ let $P_{k,i}:=P_k\boxprod A_i$ and $P_{k,i,j}:=P_k\boxprod \{v_{i,j}\}$. Note that $P_{k,i}$ is isomorphic to a $6p \times 2p$ grid. 
%  \referee{I think that for the pace of the proof, it would be nice to put right after the definition of $P_{k,i}$ the observation that $P_{k,i}$ is isomorphic to a $6p \times 2p$ grid.} \worley{Does this note flow well with the proof with respect to the reviewer comment above?} \david{this is fine, except replace $p6$ by $6p$}

  Refer to \cref{product_fig}. 
  Consider $T_i$ for some $i\in\{1,\ldots,\ell\}$. For visualization purposes, it is helpful to organize $T_i$ into $2p$ rows $T_{i,1},\ldots,T_{i,2p}$.  For each $j\in\{1,\ldots,2p-1\}$, $T_{i,j}$ and $T_{i,j+1}$ are `adjacent', in the sense that each vertex $(a,v_{i,j})\in V(T_{i,j})$ is adjacent to $(a,v_{i,j+1})\in V(T_{i,j+1})$.  We then organize $T_1,\ldots,T_{\ell}$ into a sequence of blocks. These blocks are independent in the sense that there is no edge between $T_{i}$ and $T_j$ for any $i\neq j$.  Moreover, there is an additional row $T_0$ that is adjacent to the first row, $T_{i,1}$, of each block $T_i$.

  \begin{figure}[H]
    \begin{center}
      \includegraphics[page=4]{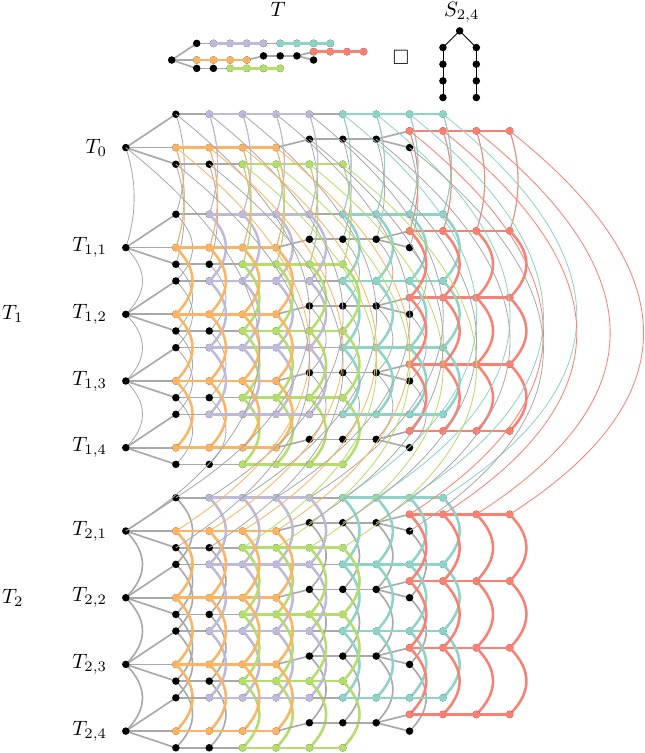}
    \end{center}
    \caption{Visualizing the product in \cref{star_times_star}}
    \label{product_fig}
  \end{figure}

  Refer to \cref{grid_partition}.  We will construct a model of $\boxplus_{sp}$. 
  We partition this model into $s^2$ subgrids each of which is isomorphic to $\boxplus_p$.  Therefore, we need $s^2$ such subgrids $G_1,\ldots,G_{s^2}$. The branch sets of each subgrid $G_i$ will include a $p\times p$ grid within the $6p\times 2p$ grid $P_{i,i}$ (which is contained in the block $T_i$).  The additional row $T_0$ will allow us to extend the branch sets of the $4p-4$ boundary vertices of the $G_i$ into $T_{i'}$ for any $i'$ and from there they can be extended so that they are adjacent to any other subgrid $G_j$.

  \begin{figure}[t]
    \begin{center}
      \includegraphics{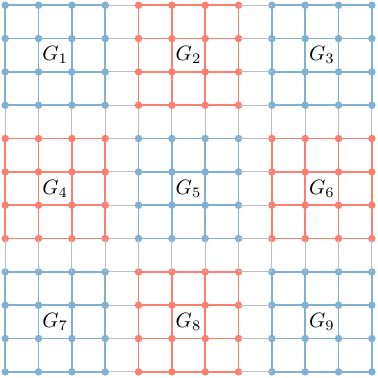}
    \end{center}
    \caption{The $sp\times sp$ grid can be partitioned into $(sp/p)^2 = s^2$ subgrids, each of which is a $p\times p$ grid. (The case $sp=12$ and $p=4$ is shown here.)}
    \label{grid_partition}
  \end{figure}

  Refer to \cref{subgrid}. To construct the branch sets for $G_i$ we start with a $p\times p$ subgrid in $P_{i,i}$ whose top row is $(p_{i,p+1},v_{i,1}),\ldots,(p_{i,2p},v_{i,1})$.  This subgrid has $4p-4$ `boundary' vertices, four of which are `corner' vertices. As illustrated in \cref{subgrid}, we create $4p$ disjoint paths in $P_{i,i}$ from the boundary vertices to $P_{i,i,1}$.  We then add one vertex of $P_{i,0}$ to each path.  In this way, the first $4p$ vertices of the path $P_{i,0}$ are partitioned into four subpaths, each of size $p$ corresponding to edges coming out of the left, top, right, and bottom of boundary of $G_i$.  The final $2p$ vertices of $P_{i,0}$ are not yet used (though we may add them to the branch sets of some boundary vertices later).

  \begin{figure}[t]
    \begin{center}
      \includegraphics[width=\textwidth]{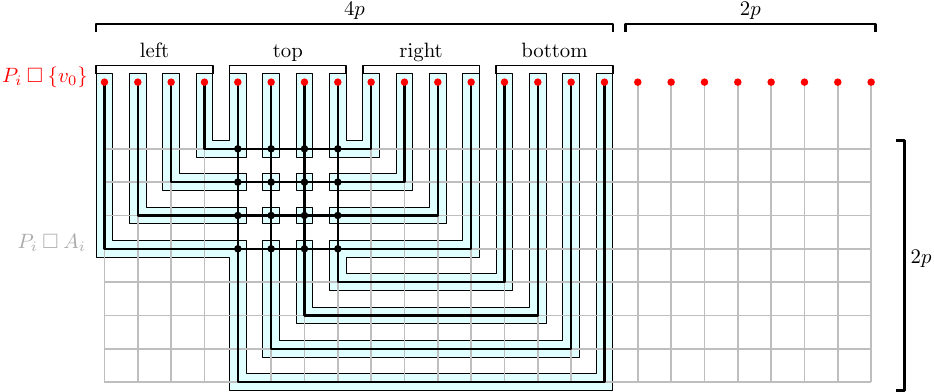}
    \end{center}
    \caption{One of the $p\times p$ subgrids used in the proof of \cref{star_times_star}.}
    \label{subgrid}
  \end{figure}

  Our model is not yet complete.  At this point, it is a model of a graph that consists of $s^2$ components, each of which is a $p\times p$ grid.  At this point, the vertices in our (not yet-complete) model are all contained in $T_0,T_1,\ldots,T_{s^2}$ and, for each $i\in\{1,\ldots,s^2\}$, the branch sets of vertices in $G_i$ are contained in $P_{i,i}\cup P_{i,0}\subseteq T_{i}\cup P_{i,0}$.  This still leaves the vertices of $T_{s^2+1},\ldots,T_{5s^2}$ unused.  We will use these to extend the branch sets of vertices on the boundary of each subgrid $G_i$ to create the required adjacencies.  For each $i\in\{1,\ldots,s^2\}$, the vertices of $T_{s^2+4i-3},\ldots,T_{s^2+4i}$ will be reserved for the branch sets of $G_i$.

% \referee{Page 11, paragraph 2, line 5: “$x′_k := p_{i, p-k+1}$” instead of “$p_{i,k}$” since the order top-to-bottom becomes a right-to-left with the rotation for the left side of the grid. And for the same reason, in page 13, all the occurrences of $p_{i,5p-k+1}$ should be $p_{i,4p+k}$ (since the new rotations put them now in a left-to-right order). This two changes should suffice to fix the proof.}
% \worley{Updated indexing accordingly, I believe the comment is correct.} \david{Pat: please check}

  First, suppose that $G_i$ is a subgrid that is immediately to the right of some subgrid $G_j$, so that the left boundary of $G_i$ is adjacent to the right boundary of $G_j$.  We will extend the branch sets for vertices on the left boundary of $G_i$ so that they become adjacent to the branch sets for vertices in the right boundary of $G_j$. Let $x_1,\ldots,x_p$ be the vertices on the left boundary of $G_i$, ordered from top to bottom and, for each $k\in\{1,\ldots,p\}$, let $x_k':=p_{i,p-k+1}$
  % \worley{Updated, previously $x_k'=p_{i,k}$} \david{Pat: please check} \pat{Checked, correct.}
  so that $(x_k',v_0)$ is already included in the branch set for $x_k$.  Let $y_0,\ldots,y_p$ be the vertices on the right boundary of $G_j$, ordered from top to bottom and, for each $k\in\{1,\ldots,p\}$, let $y_k':=p_{j,2p+k}$ so that $(y_k',v_0)$ is already included in the branch set of $y_k$. There are two cases to consider.  (We strongly urge the reader to refer to \cref{straight,reverso}.)

% \david{I think it will help the reader if the text describing Figure 6 appears on the same page as the figure. So I have added a newpage here temporarily. I suggest when the paper is finished we play with the margins to remove this big whitespace.}
%\newpage

  \begin{itemize}
    \item $P_i$ and $P_j$ are completely related (see \cref{straight}): We will extend the branch sets of $x_1,\ldots,x_p$ into $T_{s^2+4i-3}$.  For each $k\in\{1,\ldots,p\}$, we extend the branch set of $x_k$ by adding the path
    \[
      (x_k',v_{s^2+4i-3,1}),\ldots,(x_k',v_{s^2+4i-3,k}) \enspace ,
    \]
    the path in $T_{s^2+4i-3,k}$ from $(x_k',v_{s^2+4i-3,k})$ to $(y_k',v_{s^2+4i-3,k})$, and the path
    \[
      (y_k',v_{s^2+4i-3,k}),\ldots,(y_k',v_{s^2+4i-3,1}) \enspace .
    \]
    The first vertex $(x_k',v_{s^2+4i-3,1})$ of this path is adjacent to $(x_k',v_0)$, which ensures that the branch set for $x_k$ is connected.  The last vertex $(y_k',v_{s^2+4i-3,1})$ is adjacent to $(y_k',v_0)$ which ensures that the branch sets for $x_k$ and $y_k$ are adjacent.

    \begin{figure}[t]
      \begin{center}
      \includegraphics[scale=0.9]{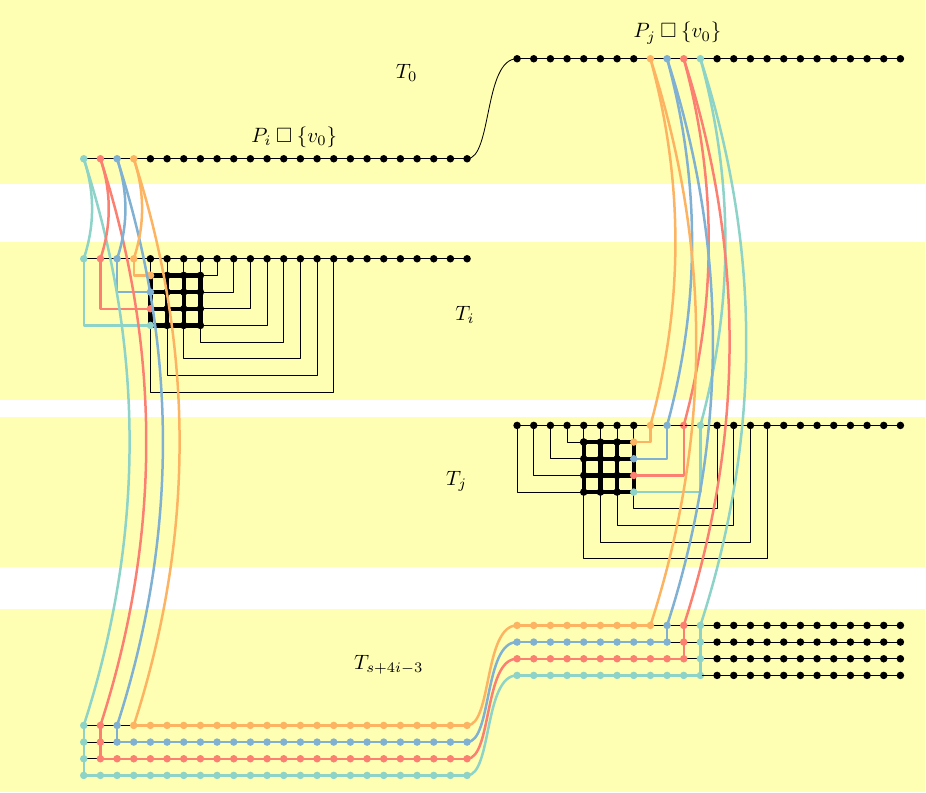}
      \end{center}
      \caption{Connecting the left side of $G_i$ to the right side of $G_j$ when $P_i$ and $P_j$ are completely related.}
      \label{straight}
    \end{figure}

    \newpage

    \item $P_i$ and $P_j$ are completely unrelated (see \cref{reverso}): We will extend the branch sets of $x_1,\ldots,x_k$ into $T_{s^2+4i-3}$ and $T_{s^2+4i-2}$. To make the connections between these two blocks we will use an additional $p$ vertices of $P_{i,0}$.  The need for a second block in this case is due to the fact that the obvious paths in $T_{s^2+4i-3}$ that were used in the previous case would either intersect each other or reverse the order of connections so that the top-left vertex of $G_i$ would become adjacent to the bottom-right vertex of $G_j$. Routing these paths through two trees allows us to make the connections in the right order using pairwise vertex-disjoint paths.

     \begin{figure}[H]
      \begin{center}
        \includegraphics[scale=0.9]{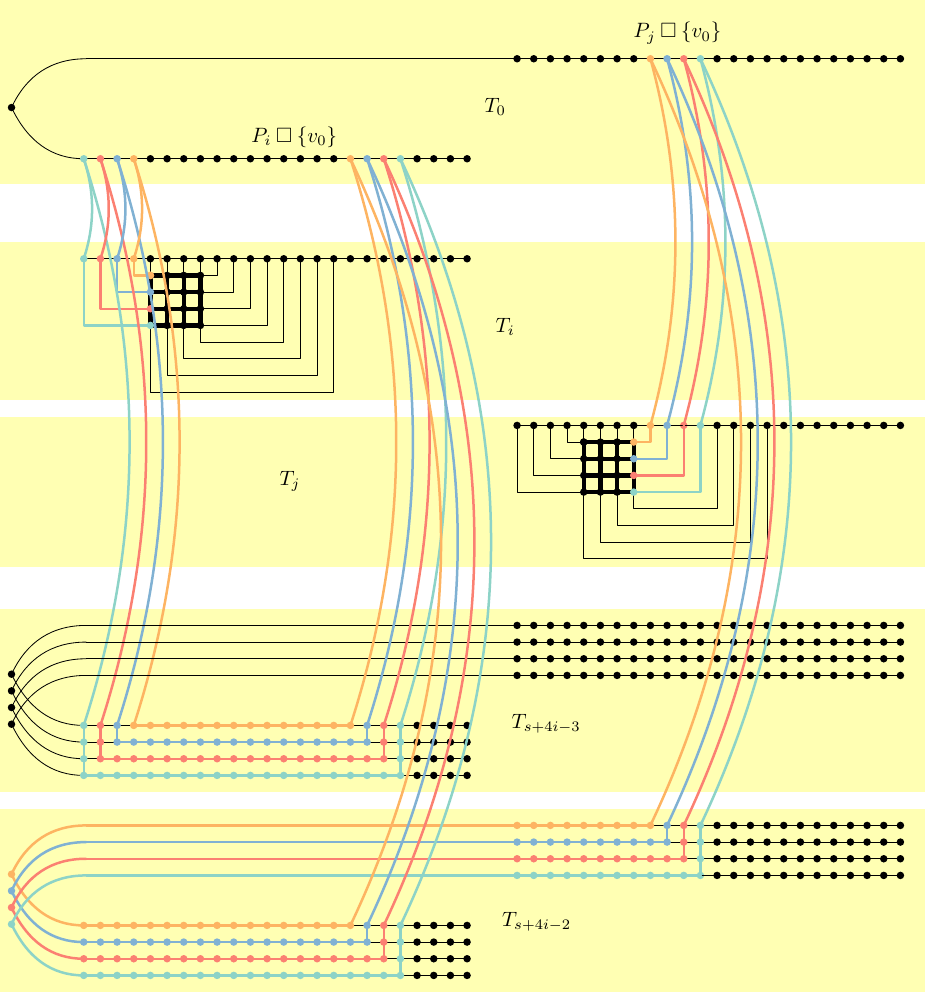}
      \end{center}
      \caption{Connecting the left side of $G_i$ to the right side of $G_j$ when $P_i$ and $P_j$ are completely unrelated.}
      \label{reverso}
    \end{figure}

    For each $k\in\{1,\ldots,p\}$ we extend the branch set of $x_k$ by adding the path,
    % \worley{Updated all instances of $p_{i,5p-k+1}$ to $p_{i, 4p+k}$ to maintain left to right order} \david{Pat: please check} \pat{Looks good.}
    \[
      (x_k',v_{s^2+4i-3,1}),\ldots,(x_k',v_{s^2+4i-3,k}) \enspace ,
    \]
    the path in $T_{s^2+4i-3,k}$ from  $(x_k',v_{s^2+4i-3,k})$ to $(p_{i,4p+k},v_{s^2+4i-3,k})$, the path
    \[
      (p_{i,4p+k},v_{s^2+4i-3,k}),\ldots,(p_{i,4p+k},v_{s^2+4i-3,0}),(p_{i,4p+k},v_{s^2+4i-2,1}),\ldots,(p_{i,4p+k},v_{s^2+4i-2,k})
    \]
    the path in $T_{s^2+4i-2,k}$ from $(p_{i,4p+k},v_{s^2+4i-2,k})$ to $(y_k',v_{s^2+4i-2,k})$ and finally the path
    \[
      (y_k',v_{s^2+4i-3,k}),\ldots,(y_k',v_{s^2+4i-3,1}) \enspace .
    \]
    As in the previous case, the first vertex of this path ensures that the branch set for $x_k$ is connected and the last vertex ensures that the branch sets of $x_k$ and $y_k$ are adjacent.
  \end{itemize}

  So far every horizontal grid edge is modelled, but not the vertical edges between $p\times p$ subgrids.  We now sketch how these can be included.  Suppose that the subgrid $G_i$ is directly below the subgrid $G_j$.  Let $x_1,\ldots,x_p$ be the top boundary of $G_i$ ordered so that $x_1$ is the the leftmost vertex and $x_p$ is the rightmost.  For each $k\in\{1,\ldots,p\}$, let $x_k':=(p_{i,p+k})$ so that the branch set of $x_k$ includes $x_k'$.  Let $y_1,\ldots,y_p$ be the bottom boundary of $G_j$ ordered so that $y_1$ is the the leftmost vertex and $y_p$ is the rightmost.  For each $k\in\{1,\ldots,p\}$, let $y_k':=(p_{j,4p-k+1})$ so that the branch set of $y_k$ includes $y_k'$.  Observe that $x_1',\ldots,x_p'$ occur in order along $P_{i}$ but $y_1',\ldots,y_k'$ occur in reverse order along $P_j$.  The effect of this is to reverse the two cases that appear above, so that the straightforward case occurs when $P_i$ and $P_j$ are completely unrelated and the more complicated case occurs when they are completely related.  Otherwise, the process of growing the branch sets for $x_1,\ldots,x_p$ is the same except that the vertices used to grow these new branch sets are contained in $T_{s+4i-1}$, $T_{s+4i}$, and $(p_{i,5p+1},v_0),\ldots,(p_{i,6p},v_0)$.  This ensures that these branch sets do not reuse vertices that are used to make $G_i$ adjacent to the neighbour on its left.
  
    % \worley{replaced $r \times r$ with $sp \times sp$} \david{fine}
    
  Checking that the resulting collection of branch sets is indeed a model of the $sp\times sp$ grid is straightforward; both the disjointedness of the branch sets and the required adjacencies are guaranteed by the construction.
\end{proof}

We now establish our lower bound on the largest grid minor in a Cartesian product:

\begin{thm}\label{lower_bound}
For any connected graphs $G_1$ and $G_2$ each having at least $n\ge 1$ vertices, 
$$\gm(G_1\boxprod G_2)\in\Omega(\sqrt{n}).$$
\end{thm}

% \worley{Updated ``for all $i\in\{1,\ldots,p_2\}$'' to  ``for all $i\in\{1,\ldots,p_2-1\}''$. I looked through and dont believe this change has any other effects on the proof} \david{Pat: please check} \pat{Looks good.}

\begin{proof}
  For each $b\in\{1,2\}$, let $T_b$ be a tree contained in $G_b$ and having exactly $n$ vertices (which can be constructed by successively deleting leaves starting with a spanning tree of $G_b$).  
  For each $b\in\{1,2\}$, let $p_b=\min\{i: n_i(T_b)\ge \tfrac{3n}{2(\pi i)^2}\}$.  (This is well-defined since, otherwise $n=\sum_{i=1}^\infty n_i(T_b) < \sum_{i=1}^\infty \tfrac{3n}{2(\pi i)^2} = \frac{n}{4}$.)  Without loss of generality, assume $p_2 \le p_1$ and let $\ell:=\ceil{\tfrac{3n}{2(\pi p_2)^2}}$. By \cref{same_height_unrelated}, $S_{\ell,p_2}\preceq T_2\preceq G_2$.  If $p_2 \le 5$ then $\ell > \frac{3n}{50\pi^2}\in\Omega(n)$ and by \cref{anything_times_star} $K_{\ell}\preceq G_1 \boxprod S_\ell$.
  Since $\boxplus_{\sqrt{\ell}}\preceq K_{\ell}$, this implies that $\gm(G_1\boxprod G_2)\ge \sqrt{\ell}=\Omega(\sqrt{n})$ and we are done, so we may assume that $p_2\ge 6$. Let $p:=\floor{p_2/6}\ge 1$.
  
  Since $p_1\ge p_2$, $n_i(T_1)\le \tfrac{3n}{2(\pi i)^2}$ for all $i\in\{1,\ldots,p_2-1\}$.  Therefore, \cref{disjoint_p_paths} implies that $T_1$ contains at least $n/4p_2$ pairwise disjoint paths $P_1,\ldots,P_{\ceil{n/4p_2}}$, each of length $p_2\ge 6p$, such that each pair of paths is either completely related or completely unrelated.  Let
  \[
    s:=\floor{\min\{\sqrt{\ell/5}, \sqrt{n/4p_2}\}} = \Theta(\sqrt{n}/p)
  \]
  so that $\ell \ge 5s^2$ and $\ceil{n/4p_2}\ge s^2$.  By \cref{star_times_star}, $\gm(T_1 \boxprod S_{\ell,6p}) \ge sp=\Theta(\sqrt{n})$.  The lemma now follows from \cref{minor_product}, the fact that $T_1\preceq G_1$, and the fact that $S_{\ell,6p}\preceq S_{\ell,p_2}\preceq G_2$.
  %
  %
  %
  %
  %   \pat{We don't immediately get our result from \cref{star_times_star} because $p=\min\{p_1,p_2\}$.  To see this, consider the case where $p_1=1$, $\ell_1=n$, $p_2=\log n$ and $\ell_2=n/\log^2 n$.  Then $p=1$ and $N=p_2\ell_2=n/\log n$ so \cref{star_times_star} only guarantees that $\gm(G_1\boxprod G_2)=\Omega(\sqrt{p N})=\Omega(\sqrt{n/\log n})$.  Of course, \cref{anything_times_star} handles this, but....}
  %
  % Right now, the best argument I can think of is to branch on the value of $p$.  Apply \cref{newer_subdivided_star_minor} to each of $G_1$ and $G_2$ and suppose that $p_1< p_2$.  If $p:=p_1 < \sqrt{\log n}$ then use \cref{anything_times_star} to deduce that $G_1\boxprod G_2\succeq K_{n/p}\succeq K_{n/\sqrt{\log n}}\succeq \boxplus_{\sqrt{n}/\log^{1/4} n}$, so $\gm(G_1\boxprod G_2)=\Omega(\sqrt{n}/\log^{1/4} n)$.  Otherwise, apply \cref{new_subdivided_star_minor} to $G_2$ to deduce that $G_2\succeq S_{\ell_2,p_2}$ with $\ell_2p_2 \ge n/\log n$ (and $p_2\ge \log^{1/4} n$?).  Then apply \cref{star_times_star} to deduce that $\gm(G_1\boxprod G_2)=\Omega(\sqrt{p\ell_2p_2}) \ge \sqrt{n/\sqrt{\log n}}=\Omega(n/\log^{1/4} n)$.
  % Let $\ell = \min\{\ell_1, \ell_2\}$ and $p = \min\{p_1, p_2\}$. Then, by Lemma~\ref{star_times_star}, $\gm(S_{\ell_1,p_1}\boxtimes S_{\ell_2,p_2})\in\Omega(\sqrt{pN}) = \Omega(\sqrt{\ell p^2}) = \Omega(\sqrt{n})$. Since $S_{\ell_b,p_b}\preceq T_b\preceq G_b$ for each $b\in\{1,2\}$, the result follows from two applications of Observation~\ref{minor_product}.
\end{proof}

Our next result completes the relationships between grid minors and treewidth in Cartesian and strong products of trees.  

\begin{thm}\label{quadratic_grid_minor}
    For any two trees $T_1$ and $T_2$, 
    \[
\gm(T_1\boxprod T_2) \leq
\gm(T_1\boxtimes T_2) \leq \tw(T_1\boxtimes T_2) \in O(\gm(T_1\boxprod T_2)^2)
 \enspace .
\]
\end{thm}

\begin{proof}
    First note that $\gm(T_1 \boxprod T_2) \le \gm(T_1 \boxtimes T_2)$ since  
    $T_1 \boxprod T_2 \subseteq T_1 \boxtimes T_2$. \cref{tw_gte_gm} shows that 
    $\gm(T_1 \boxtimes T_2) \le \tw(T_1\boxtimes T_2)$. It remains to show that $\tw(T_1\boxtimes T_2) \in O(\gm(T_1 \boxprod T_2)^2)$.
    
    Let $n_1:=|V(T_1)|$, let $n_2:=|V(T_2)|$, and assume without loss of generality that $n_1\le n_2$. By \cref{EasyUpperBound}, 
    %Root $T_2$ at an arbitrary vertex.  Then $T_1\boxtimes T_2$ has a tree decomposition $\mathcal{T}:=(B_x:x\in V(T_2))$ where $B_x$ contains $V(T_1)\times\{x\}$ and, if $x$ has parent $y$, then $B_x$ also contains $V(T_1)\times\{y\}$.  This decomposition has width at most $2n_1-1$, so 
    $\tw(T_1\boxtimes T_2)\le 2n_1-1$.  By \cref{lower_bound}, $c\cdot\gm(T_1\boxprod T_2) \ge \sqrt{2n_1}$ for some fixed positive constant $c$. 
    Therefore, 
    \[
       (c\cdot\gm(T_1\boxprod T_2))^2 \ge 2n_1 > \tw(T_1\boxtimes T_2) \enspace . \qedhere
    \]
\end{proof}

%\david{Can we say that each of the inequalities in \cref{quadratic_grid_minor} is tight up to a constant factor for certain trees $T_1$ and $T_2$ (depending on the inequality)?} \pat{I think so.  The first one is tight by \cref{StarTreeUpperBound,lower_bound}, so $\gm(S_n\boxprod T)=\Theta(\sqrt {n})$ and $\gm(S_n\boxtimes T)=\Theta(\sqrt{n})$.  The second one is tight because $\gm(P_n\boxtimes P_n)=\Theta(n)$ and $\tw(P_n\boxtimes P_n)=\Theta(n)$.  The third one is tight because $\tw(S_n\boxtimes P_n)=\Theta(n)$ and $\gm(S_n\boxprod P_n)=\Theta(\sqrt{n})$.  }

It is worth pointing out that each of the inequalities in~\cref{quadratic_grid_minor} is tight for certain trees $T_1$ and $T_2$. The first two inequalities are tight for the product of two paths. Specifically, it is obvious that $\gm(P_n\boxprod P_n)=\gm(P_n\boxtimes P_n)=n$, and $\tw(P_n\boxtimes P_n)<2n$ by \eqref{TreewidthProductTrees}. 
%\cref{StarTreeUpperBound} shows that for any $n$-vertex tree $T$, $\gm(S_n\boxprod T) \in \Theta(\sqrt{n})$, and $\gm(S_n \boxtimes T) \in \Theta(\sqrt{n})$  by \cref{lower_bound} %\david{should this be $\Theta(\sqrt{n})$?}. 
% The second inequality is tight for the product of two paths. Specifically, it is well-known that $\gm(P_n \boxtimes P_n) \in \Theta(n)$ and $\tw(P_n \boxtimes P_n) \in \Theta(n)$. \david{It would be simpler to use two paths for the first two inequalities. Obviously, $\gm(P_n\boxprod P_n)=\gm(P_n\boxtimes P_n)=n$, and $\tw(P_n\boxtimes P_n)<2n$ by 
% \eqref{TreewidthProductTrees}.
 The last inequality is tight for $S_n$ and $P_n$ since $\tw(S_n \boxtimes P_n) \in \Theta(n)$ by \eqref{LowerBoundProduct} and \cref{EasyUpperBound}, and $\gm(S_n \boxprod P_n) \in \Theta(\sqrt{n})$ by \cref{lower_bound,StarTreeUpperBound} below. 

% \pat{Can we strengthen \cref{quadratic_grid_minor} so that it applies to any graphs $G_1$ and $G_2$?  Answer: Not without improving the grid-minor theorem.  The proof of \cref{quadratic_grid_minor} does prove that $\tw(G\boxtimes T)\in O(\gm(G\boxtimes T))$ when $G$ is any graph, $T$ is a tree, and $|V(G)|\le |V(T)|$. More generally, it works when the bigger graph, $G_2$, has bounded treewidth.} 

\section{Upper Bound}\label{C}

This section proves upper bounds of the form, $\gm(G)\in O(\sqrt{n})$, where $G$ is the product of various $n$-vertex graphs, as mentioned in \cref{A}.

\begin{lem}
\label{StarGraph}
Fix numbers $\Delta\geq c>0$. Let $\GG$ be a  graph class closed under minors and disjoint unions, such that $|E(H)|<c|V(H)|$ for every graph $H\in\GG$. Let $S$ be any star and $H$ be any graph in $\GG$. Let $G$ be any graph with maximum degree $\Delta$ that is a minor of $S \bigcdot H$. Then 
$$|E(G)|< c|V(G)|+(\Delta-c)\,|V(H)|.$$ 
\end{lem}

\begin{proof}
Let
$(B_x:x\in V(G))$ be a model of $G$ in $S \bigcdot H$. Let $r$ be the root of $S$. Let $R$ be the set of vertices $x$ of $G$ such that $(r,b)\in V(B_x)$ for some $b\in V(H)$. 
Let $Q$ be the set of vertices $x$ of $G$ such that $V(B_x)\subseteq \{(v,b): v\in V(S-r),b\in V(H)\}$. 
Thus $\{R,Q\}$ is a partition of $V(G)$. Moreover, $G[Q]$ is a minor of the disjoint union of $n$ copies of $H$, implying $G[Q]\in\GG$ and $|E(G[Q])|<c|Q|$. The number of edges of $G$ incident to $R$ is at most $\Delta|R|$. 
Thus $|E(G)| < c|Q| + \Delta|R| 
= c(|V(G)|-|R|)+ \Delta\,|R|
= c\,|V(G)|+ (\Delta-c)\,|R|
\leq c|V(G)|+(\Delta-c)|V(H)|$.
\end{proof}

The class of graphs with treewidth at most $t$ is closed under minors and disjoint unions, and $|E(H)|<t\,|V(H)|$ for every graph $H$ with treewidth at most $t$. \cref{StarGraph} implies:

\begin{cor}
\label{StarTreewidth}
Fix numbers $\Delta\geq t\geq 1$. Let $S$ be any star and $H$ be any graph with treewidth at most $t$. Let $G$ be any graph with maximum degree $\Delta$ that is a minor of $S \bigcdot H$. Then 
$$|E(G)|< t|V(G)|+(\Delta-t)\,|V(H)|.$$ 
\end{cor}

The next result completes the proof of the second part of our main theorem stated in \cref{Intro}, showing that the lower bound in \cref{lower_bound} is optimal.

\begin{thm}
\label{StarTreeUpperBound}
For any star $S$ and any $n$-vertex tree $T$, 
$$\gm(S\boxprod T) \leq \gm(S \boxtimes T) \leq \gm(S\bigcdot T) <\sqrt{3n+1}+1.$$
\end{thm}

\begin{proof}
The first two inequalities hold by definition. Let $k:=\gm(S\bigcdot T)$. Now apply \cref{StarTreewidth} with $t=1$, and with $G:=\boxplus_k$ and $\Delta=4$. Thus
$$2k(k-1)  = |E(G)| < |V(G)| + (\Delta-1)|V(T)|
= k^2 + 3 n.
$$ 
Thus $k^2 -2k < 3n$ and $k<\sqrt{3n+1}+1$.
\end{proof}

We now show that results like \cref{StarTreeUpperBound} can be concluded from results in the literature. (We include the above proofs for the sake of completeness and since \cref{StarGraph,StarTreewidth} are of independent interest.) 

A set $S$ of vertices in a graph $G$ is a \defn{feedback vertex set} if $G-S$ is a forest. \citet{Luccio98} proved that the minimum size of a feedback vertex set in $\boxplus_k$ is $(\tfrac13+o(1))k^2$. If $\boxplus_k$ is a minor of $S\bigcdot T$ with $|V(T)|=n$, then $\boxplus_k$ has a feedback vertex set of size at most $n$ (consisting of the vertices of $\boxplus_k$ whose branch sets intersect the copy of $T$ corresponding to the root of $S$). Thus $n\geq (\tfrac13+o(1))k^2$ and $k\leq (1+o(1))\sqrt{3n}$, which implies \cref{StarTreeUpperBound}.

A result similar to \cref{StarTreeUpperBound} can also be concluded from a more general result of~\citet[Theorem~3]{Eppstein14}, who proved that if $n>k/2$ and any set of $n$ vertices are deleted from $\boxplus_k$, then the remaining graph contains a $\boxplus_\ell$ minor, where $\ell\geq \frac{k^2}{4n}-1$. Say $k=\gm(S\bigcdot T)$ where $S$ is a star and $T$ is any $n$-vertex tree. By the definition of $S\bigcdot T$, a set of at most $n$ vertices can be deleted from $\boxplus_k$ (corresponding to branch sets that intersect the copy of $T$ at the root of $S$) so that each component of the remaining graph is a tree (a minor of $T$), and thus contains no $\boxplus_2$ minor. 
Note that $k \leq \tw(S\bigcdot T)\leq 2n-1$, so Eppstein's result is applicable. Hence  $1 \geq \ell\geq \frac{k^2}{4n}-1$, implying $k\leq\sqrt{8n}$. Eppstein's result is substantially stronger than \cref{StarTreeUpperBound}. For example, it implies: 

\begin{thm}
\label{StarGraphUpperBound}
For any star $S$ and any $n$-vertex graph $H$, 
$$\gm(S\bigcdot H)\leq 2\sqrt{n( \tw(H) +1)}.$$
\end{thm}

\cref{StarGraphUpperBound} can be generalised to allow for arbitrary trees with bounded radius.

\begin{thm}
\label{TreeGraphUpperBound}
For any tree $T$ with radius $r$ and any $n$-vertex graph $H$ with $E(H)\neq\emptyset$, 
$$\gm(T\bigcdot H)\leq 5 n^{1-1/2^r} \tw(H)^{1/2^r}.$$
\end{thm}

\begin{proof}
We proceed by induction on $r\geq 0$ (with $H$ fixed). In the $r=0$ case, $T=K_1$ and
$$\gm(T\bigcdot H)\leq \tw(H) = n^{1-1/2^r} \tw(H),$$ as desired. 
Now assume $r\geq 1$ and the result holds for $r-1$. Let $T$ be any tree $T$ with radius $r$. Let $G:=T\bigcdot H$ and let $k:=\gm(G)$. Let $v$ the centre of $T$. Let $T_1,\dots,T_p$ be the components of $T-v$. Let $v_i$ be the neighbour of $v$ in $T_i$. So $v_i$ has eccentricity at most $r-1$ in $T_i$, and $T_i$ has radius at most $r-1$. Let $k_i:=\gm(T_i\bigcdot H)$. By induction, for each $i\in\{1,\dots,p\}$,
$$\gm(T_i\bigcdot H)\leq 5 n^{1-1/2^{r-1}} \tw(H)^{1/2^{r-1}}.$$
By the definition of $T\bigcdot H$, there is a set $X$ of at most $n$ vertices in $G$ (corresponding to branch sets that intersect the copy of $H$ at $v$) such that each component of $G-X$ is a minor of some $T_i\bigcdot H$. 
By Eppstein's result (which is applicable since $k\leq \tw(G) \leq 2n-1$), 
$$ \frac{k^2}{4n}-1 \leq \gm(G-X) \leq 
\max_i \gm(T_i\bigcdot H) \leq 5 n^{1-1/2^{r-1}} \tw(H)^{1/2^{r-1}}.$$
Since $n\geq 1$ and $\tw(H)\geq 1$, 
$$ \frac{k^2}{4n} \leq 6 n^{1-1/2^{r-1}} \tw(H)^{1/2^{r-1}}$$
and
$$ k \leq 
\sqrt{24} n^{1-1/2^r} \tw(H)^{1/2^r}
< 5 n^{1-1/2^r} \tw(H)^{1/2^r}
.$$
The result follows.
\end{proof}

\section{Product of Stars and Trees}\label{D}

Given the importance of products of stars and trees in the previous section, we  now consider the following natural question: 
What is the least constant $c$ such that 
for any star $S$ and any $n$-vertex tree $T$,
\begin{equation}
\label{StarTreeQuestion}
\gm(S\boxprod T) \leq(1 +o(1))\sqrt{c n}\,?
\end{equation}
Analogous questions are interesting for $S\boxtimes T$ and $S\bigcdot T$. 
\cref{StarTreeUpperBound} gives an upper bound of $c \leq 3$ for $S\boxprod T$, $S\boxtimes T$ or $S\bigcdot T$.
For Cartesian products, we now show that $c =2$ is the answer. 

\begin{lem}
\label{BipartiteG}
    Let $G$ be a bipartite graph with bipartition $\{A,B\}$. 
    Let $S$ be a star with at least $|A|$ leaves. 
    Let $T$ be any tree with at least $|B|$ vertices. 
    Then $G$ is a minor of $S\boxprod T$.
\end{lem}
\begin{proof}
Let $r$ be the root of $S$. 
Let $f$ be an injection from $A$ into the leaf-set of $S$. 
Let $g$ be an injection from $B$ into the vertex-set of $T$.
For each vertex $v\in A$, define the branch set $B_v:= \{f(v)\}\times V(T)$, which induces a connected copy of $T$ in $S\boxprod T$.
For each vertex $w\in B$, define the branch set $B_w := \{(r,g(w))\}$. For each edge $vw$ of $G$ with $v\in A$ and $w\in B$, the edge $(f(v),g(w))(r,g(w))$ of $S\boxprod T$ joins $B_v$ and $B_w$. Hence $(B_v:v\in V(G))$ is a model of $G$ in $S\boxprod T$. 
\end{proof}

\begin{cor}\label{uppercorr}
For any tree $T$ with $n$ vertices, and 
for any star $S$ with at least $n+1$ leaves, 
$$\gm(S\boxprod T)\geq \floor{\sqrt{2n}}.$$
\end{cor}

\begin{proof}
Let $k:= \floor{\sqrt{2n}}$. 
Let $\{A,B\}$ be the bipartition of $\boxplus_k$, where $|A|=\ceil{\frac{k^2}{2}}$ and $|B|=\floor{\frac{k^2}{2}}$. 
So $S$ has at least $n+1\geq \frac{k^2+1}{2} \geq |A|$ leaves, and  $T$ has at least $n\geq \frac{k^2}{2} \geq |B|$ vertices. By \cref{BipartiteG}, $\boxplus_k$ is a minor of $S\boxprod T$, and $\gm(S\boxprod T)\geq k$. 
\end{proof}

\begin{lem}\label{star_times_tree_cartesian}
  For any star $S$ and any $n$-vertex tree $T$, 
  $$\gm(S\boxprod T)\le \sqrt{2n}+1.$$
\end{lem}

\begin{proof}
  Let $\mathcal{M}:=\{B_x:x\in V(\boxplus_k)\}$ be a model of $\boxplus_k$ in $G:=S\boxprod T$.  Let $v_0,\ldots,v_\mu$ denote the vertices of $S$, where $v_0$ has degree $\mu$ and $v_1,\ldots,v_\mu$ are leaves.  For each $i\in\{0,\ldots,\mu\}$ let $R_i:=\{v_i\}\boxprod V(T)$, so that $T_i:=G[R_i]$ is isomorphic to $T$.  The idea behind the rest of this proof is that the vertices in $R_0$ are special because each vertex in $R_0$ is adjacent to the corresponding vertex in each of $T_1,\ldots,T_\mu$. This makes the vertices in $R_0$ a scarce resource that the model $\mathcal{M}$ must use efficiently.  Let $X_0:=\{x\in V(\boxplus_k):B_x\cap R_0\neq\emptyset\}$ be the set of vertices in $\boxplus_k$ whose branch sets intersect $R_0$.

%\referee{$X$ appears twice but is not defined} \david{fixed}

  Observe that $G-R_0$ is the vertex-disjoint union of trees $T_1,\ldots,T_\mu$.  Therefore, for each component $C$ of $\boxplus_k-X_0$, $G[\bigcup_{x\in V(C)} B_x]$ is a subtree of $T_i$ for some $i\in\{1,\ldots,\mu\}$.  Let $\mathcal{F}$ be the set of $4$-cycles in $\boxplus_k$. (So $\mathcal{F}$ is the set of inner faces in the usual plane drawing of $\boxplus_k$.)  Observe that, for each $F\in\mathcal{F}$, $|V(F)\cap X_0|\ge 1$ since otherwise $F$ would be a minor of $G-R_0$, which is not possible since $F$ is a cycle and $G-R_0$ is a forest.  
  
  Next we show that $|\bigcup_{x\in V(F)} B_x\cap R_0|\ge 2$ for each $F\in\mathcal{F}$.  If $|V(F)\cap X_0|\ge 2$, then this is immediate, so suppose that $|V(F)\cap X_0|=1$.  
  %Refer to \cref{star_times_tree_cartesian_fig}.  
  Let $F:=xx_1x_2x_3$ where $x\in X_0$ and $x_1,x_2,x_3 \in V(\boxplus_k)\setminus X_0$.  Since $G[\{x_1,x_2,x_3\}]$ is connected, there exists an $i\in\{1,\ldots,n\}$ such that $T'_j:=G[B_{x_j}]$ is a subtree of $T_i$ for each $j\in\{1,2,3\}$.  Furthermore, the subtree $T'_2$ is `between' $T'_1$ and $T'_3$, in the sense that $B_{x_1}$ and $B_{x_2}$ are in different components of $T_i-B_{x_2}$.  Since $N_{\boxplus_k}(x)$ contains both $x_1$ and $x_3$, $B_x$ contains vertices in $N_{G}(B_{x_1})$ and in  $N_{G}(B_{x_3})$.  Since $G[B_x]$ is connected, this implies that $G[B_x]$ contains a path from some vertex in  $N_{G}(B_{x_1})$ to some vertex in $N_{G}(B_{x_3})$.  Since $G[B_x]$ is a subgraph of $G-B_{x_2}$, this implies that any such path must contain at least two vertices of $R_0$.\footnote{This is true for the Cartesian product $S\boxprod T$, but not true for the strong product $S\boxtimes T$, since $N_{S\boxtimes T}(B_{x_1})$ and $N_{S\boxtimes T}(B_{x_3})$ can have a common vertex in $T_0$. This fact is used in the construction described in \cref{star_times_tree_strong}.}  Thus $|\bigcup_{x\in V(F)}B_x\cap R_0|=|B_x\cap R_0|\ge 2$, as claimed.

%\referee{Figure 8 is not much helpful; perhaps an expanded caption can help.} \david{I would delete the figure} \pat{Agreed}
%  \begin{figure}[ht]
%    \centering
%    \includegraphics{figs/star_times_tree_cartesian}
%    \caption{The proof of \cref{star_times_tree_cartesian}.}
%    \label{star_times_tree_cartesian_fig}
%  \end{figure}

  Now we can finish the proof with
  \[
    \sum_{F\in\mathcal{F}}\sum_{x\in V(F)} |B_x\cap R_0| \ge \sum_{F\in\mathcal{F}} 2 = 2|\mathcal{F}| = 2(k-1)^2 \enspace .
  \]
  Each vertex in $R_0$ appears in $B_x$ for at most one $x\in V(\boxplus_k)$, and each such $x$ appears in $V(F)$ for at most four cycles $F\in\mathcal{F}$, so 
  \[
    \sum_{F\in\mathcal{F}}\sum_{x\in V(F)} |B_x\cap R_0| \le 4\sum_{x\in V(\boxplus_k)} |B_x\cap R_0| \le 4|R_0| = 4n \enspace .
  \]
  Combining the previous two equations gives $4n\ge 2(k-1)^2$, so $k \le \sqrt{2n}+1$.
\end{proof}

\cref{uppercorr,star_times_tree_cartesian} together show that $c=2$ in \eqref{StarTreeQuestion}. 

Now consider \eqref{StarTreeQuestion} for the strong product $S\boxtimes T$. We now show that the answers for Cartesian and strong products are different. In particular, for $S\boxtimes T$ the minimum $c$ in \eqref{StarTreeQuestion} satisfies $\frac52 \leq c \leq 3$. 

\begin{lem}\label{star_times_tree_strong}
  For any star $S$ with at least $\sqrt{10(n-2)}+1$ leaves and any path $P$ on at least $n$ vertices, $$\gm(S\boxtimes P)\ge \lfloor\sqrt{5(n-2)/2}\rfloor$$
\end{lem}

\begin{proof}
  Let $k:=\lfloor\sqrt{5(n-2)/2}\rfloor$ and note that $n\ge 2k^2/5 + 2$.
  Let $G:=S\boxtimes P$. 
  Recall that $V(\boxplus_k)=\{1,\ldots,k\}^2$.  For each $i\in\{0,\ldots,2k-2\}$, let $D_i:=\{(x,y)\in V(\boxplus_k):y-x=k-i-1\}$,  and for convenience let $D_i=\emptyset$ for any $i\not\in\{0,\ldots,2k-2\}$. %\david{Not that it matters, but would it be cleaner to define $D_i:=\{(x,x-i)\in V(\boxplus_k):x\in[k]\}$ for $i\in\{1-k,\dots,k-1\}$? \pat{I tried, but then the indexing gets very confusing below.}} \david{okay}
  (The vertices of each $D_i$ are contained in a line of slope $1$.)  Let the vertices of $S$ be $v_0,\ldots,v_{2k+1}$, where $v_0$ is the vertex of degree $2k+1$ and  $v_1,\ldots,v_{2k+1}$ are the leaves.  
  % For each $i\in\{0,\ldots,2k+1\}$, let $P_i:=G[\{v_i\}\boxprod V(P)]$.

  For each $\ell\in\{0,\ldots,4\}$, let $\alpha_\ell:=|\bigcup_{a\in\mathbb{Z}} D_{\ell+5a+1}\cup D_{\ell+5a+2}|$.  Observe that $\sum_{\ell=0}^4 \alpha_\ell=2|V(\boxplus_k)|=2k^2$, since each diagonal $D_i$ contributes to this sum exactly twice (when $i-\ell\equiv 1\pmod 5$ and when $i-\ell\equiv 2\pmod 5$).  Therefore, $\alpha_\ell\le 2k^2/5$ for some $\ell\in\{0,\ldots,4\}$.  For the sake of brevity, assume that $\alpha_0\le 2k^2/5$.  (The only reason for this assumption is to avoid having to include an $\ell$ term in many of the subscripts in the following paragraphs.)
  
  % Let the vertices of $P$ be $w_1,\ldots,w_n$, in the order they appear along $P$.  

  We now explain how to embed a small (width-$4$) strip $B_a$ of $\boxplus_k$ into a small subgraph $S\subseteq G[\{v_0,v_1\}\times V(P)]$ of $G$.  This will allow us to decompose $\boxplus_k$ into disjoint independent strips and embed them all on $G[\{v_0,v_1\}\times V(P)]$.  We will then complete the embedding by embedding the remaining (width-$1$) strips into $G[\{v_2,\ldots,v_{k+1}\}\times V(P)\}]$.  Fix some integer $a\in\{0,\ldots,\lfloor 2k/5\rfloor-1\}$, and consider the induced subgraph $B_a:=\boxplus_k[D_{5a}\cup\cdots \cup D_{5a+3}]$.  Let $S':=S[\{v_0,v_1\}]$ and let $P_a'$ be any  subpath of $P$ with $|D_{5a+1}\cup D_{5a+2}|$ vertices.  As illustrated in \cref{fivehalves}, $G':=S'\boxtimes P_a'$ contains a subgraph isomorphic to $B_a$, where the isomorphism $\varphi:B_a\to G'$ maps the vertices of $D_{5a}\cup D_{5a+3}$ onto the vertices in $\{v_0\}\times V(P_a')$.  (Note that this makes use of the fact that $|D_{5a}\cup D_{5a+3}|\le |D_{5a+2}\cup D_{5a+2}|$, valid for all $a\in\{0,\ldots,2k-5\}$.  Since $P_a'$ uses only $|D_{5a+1}\cup D_{5a+2}|$ vertices of $P$, this immediately implies that $G[\{v_0,v_1\}\times V(P)]$ contains a subgraph isomorphic to $X:=\boxplus_k-\bigcup_{a\ge 0}D_{5a+4}$.\footnote{Recall that $P$ has $n\ge 2k^2/5+2\ge \alpha_0+2$ vertices. The additional two vertices are required for two boundary cases where we can only guarantee that $|D_{\ell+5a+1}\cup D_{\ell+5a+2}|+1\ge |D_{\ell+5a}\cup D_{\ell+5a+3}|$.  These cases can occur when $\ell+5a=-2$ (because $|D_{-1}\cup D_0|=1$ and $|D_{-2}\cup D_{1}|=2$ and when $\ell+5a=2k-3$ (because $|D_{2k-2}\cup D_{2k-1}=1$ and $|D_{2k-3}\cup D_{2k}|=2$).} Furthermore, the vertex disjoint subpaths $P'_a$, $a\in \mathbb{Z}$ can be chosen so that $P'_a$ and $P'_{a+1}$ are always consecutive in $P$.
  
  \begin{figure}[ht]
    \centering
    \includegraphics{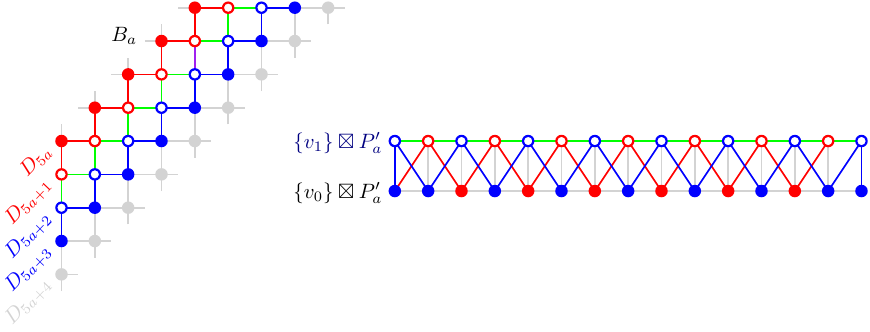}
    \caption{The construction in \cref{star_times_tree_strong}. }
    \label{fivehalves}
  \end{figure}

  By setting $B_x:=\{\varphi(x)\}$ for each $x\in V(X)$ we immediately obtain a model $\{B_x:x\in V(X)\}$ of $X$ in $G$.  We can complete this model of $X$ to a model of $\boxplus_k$ by mapping, for each $a\equiv 4\pmod 5$, the at most $k$ vertices in $D_{5a+4}$ to subsets of $\{v_2,\ldots,v_{2k+1}\}\times P$.  More precisely, let $x_1,\ldots,x_r$ be the vertices in $D_{5a+4}$.  These vertices have neighbours in $D_{5a+3}$ and in $D_{5(a+1)}$.  The vertices in $D_{5a+3}$ have branch sets in $\{v_0\}\times V(P'_a)$.  The vertices in $D_{5(a+1)}$ have branch sets in $\{v_0\}\times V(P'_{a+1})$.  If $a$ is even then we set $B_{x_i}:=\{v_{2i}\}\times (V(P'_a)\cup V(P'_{a+1}))$.  If $a$ is odd then we set $B_{x_i}:=\{v_{2i+1}\}\times (V(P'_a)\cup V(P'_{a+1}))$.  This ensures that $B_{x_i}$ is adjacent to the branch sets of $x_i$'s neighbours in $D_{5a+3}$ and $D_{5(a+1)}$.  Using different subsets of $v_2,\ldots,v_{2k+1}$ for odd and even values of $a$ ensures that the branch sets of vertices in $D_{5a+4}$ are disjoint from those of vertices in $D_{5(a+1)+4}$.  This completes the proof.
\end{proof}

% \worley{would the rootless vertices in $D_3$ at the end of the diagonals only be adjacent to one rooty vertex, the adjacent one from $D_2$?}

% \worley{The diagonals WWgive $10\floor{k/5}^2 + 4\floor{k/5}$ rooty vertices}

To complete this section, we now show that for lexicographic products $S\bigcdot T$, the answer to \eqref{StarTreeQuestion} is $c=3$. By \cref{StarTreeUpperBound}, it suffices to prove the following (where $P_3=S_2$ is the star with two leaves):

\begin{lem}
$\boxplus_k$ is isomorphic to a subgraph of $P_3 \bigcdot P_n$, where $k=\floor{\sqrt{3n-2}}$.
\label{LexST}
\end{lem}

\begin{proof}
Recall that $\boxplus_{k}$ has vertex-set $\{1,\dots,k\}^2$. 
For each $i\in\mathbb{Z}$, let $D_i:=\{(x,x+i):x,x+i\in\{1,\dots,k\}\}$. Each $D_i$ is an independent set in $\boxplus_{k}$ contained in a diagonal line of slope $1$. 
As illustrated in \cref{LexSTfig}, 
let $S:=\bigcup\{D_{i}:i\equiv 0\pmod{3}\}$, 
which is an independent set in $\boxplus_k$. 
Let $A:=\bigcup\{D_{i}:i\not\equiv 0\pmod{3}, i>0\}$ and
 $B:=\bigcup\{D_{i}:i\not\equiv 0\pmod{3}, i<0\}$. 
Each of $A$ and $B$ induce linear forests in $\boxplus_k$.  
Note that $S,A,B$ partitions $V(\boxplus_k)$, and $\floor{k^2/3}\leq|S|,|A|,|B|\leq\ceil{k^2/3}\leq n$. 
Consider the 3-vertex path $P_3=(a,s,b)$. 
Injectively map $S$ to the the copy of $P_n$ corresponding to $s$. Injectively and homomorphically map $A$ to the copy of $P_n$ corresponding to $a$. Injectively and homomorphically map $B$ to the copy of $P_n$ corresponding to $b$. This defines an injection from $V(\boxplus_k)$ to $V(P_3\bigcdot P_n)$. Since there is no edge of $\boxplus_k$ between $A$ and $B$, each edge of $\boxplus_k$ is mapped to an edge of $P_3\bigcdot P_n$, and $\boxplus_k$ is isomorphic to a subgraph of $P_3\bigcdot P_n$.
\end{proof}

\begin{figure}[ht]
\centering
\includegraphics[scale=0.7]{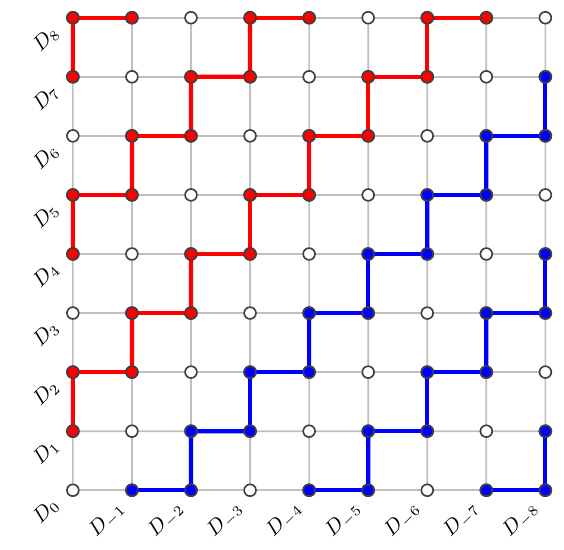}
\caption{The proof of \cref{LexST} with $k=9$: vertices in $S$ are white, red paths are in $A$, and blue paths are in $B$.}
\label{LexSTfig}
\end{figure}

\section{Open Problems}\label{E}

%\pat{I suggest we remove the following paragraph:} \david{I agree.}  \cref{quadratic_grid_minor} shows the existence of very simple graph products that do not have the linear (or even subquadratic) grid minor property.  Thus, the SQGM Framework for optimization problems on graphs with the subquadratic grid minor property does not immediately apply \david{The SQGM Framework needs to be explained before we say it does not apply.}.  Nevertheless, these graph products are still highly structured, so there may be a general framework for attacking these types of problems on (subgraphs of) graph products.  If so, this will require new ideas: The prototypical problem for the SQGM Framework is \textsc{FeedbackVertexSet} which asks for the smallest $S\subseteq V(G)$ such that $G-S$ is a forest. A critical property of the SQGM framework is that the size of an optimal solution for a graph of treewidth $k$ should be $\Omega(k^{1+\delta})$ for some $\delta>0$.  Unfortunately, $S_n\boxprod P_n$ has treewidth $2n-1$ (or so) but has a feedback vertex set of size $n$ (every copy of the root of $S_n$).

A first open problem is to tighten the bounds for $\gm(S\boxtimes T)$ presented in \cref{StarTreeUpperBound} and \cref{star_times_tree_strong}, which combine to show that $\sqrt{5(n-2)/2} \le \gm(S\boxtimes T) \le \sqrt{3n+1}+1$.
This would fully resolve the discussions resulting from \eqref{StarTreeQuestion}.

%\david{Do we want to add the following open problem (please think about counterexamples first): For any trees $T_1$ and $T_2$, are  $\gm(T_1\boxprod T_2)$, $\gm(T_1\boxtimes T_2)$ and $\gm(T_1\cdot T_2)$ within a constant factor of each other?}

Another area of future work is to further investigate the Planar Graph Product Structure Theorem, which we recall states that for every planar graph $G$, there exists a graph $H$ of bounded treewidth and a path $P$ such that $G \subsetsim H \boxtimes P$. A specific area to investigate is identifying which properties of $G$ can be preserved in $H\boxtimes P$.
%~\cite{DJMMUW20}\david{For simplicity, write ``of $G$ can be preserved in $H\boxtimes P$''. Do we need to cite \cite{DJMMUW20} here?}
Several results of this type are known. For example, in the proof of \citet{DJMMUW20}, $H$ is a minor of $G$, and so $H$ is planar. An impossibility result in this area is the following: Even if $G$ is planar and has maximum-degree $5$, a result of the form $G \subsetsim H\boxtimes P$ cannot guarantee that $H$ has bounded treewidth and bounded degree~\cite{DJMMW24}.

A concrete question that remains open is whether the treewidth of $G$ can be preserved in the product:
Is it true that for every planar graph $G$, there exists a bounded treewidth graph $H$ and a path $P$ such that $G\subsetsim H\boxtimes P$ and $\tw(H\boxtimes P) \in O(\tw(G))$? Note that 
$$\Omega(\min\{|V(H)|,|V(P)|\})\leq  \tw( H\boxtimes P) \leq O(\min\{|V(H)|,|V(P)|\}).$$ 
This upper bound follows from \cref{EasyUpperBound} since both $H$ and $P$ have bounded treewidth. The lower bound follows from \eqref{LowerBoundProduct} since we may assume that $G$, $H$ and $P$ are connected. So this question really asks whether for every planar graph $G$, there exists a bounded treewidth graph $H$ and a path $P$ such that $G\subsetsim H\boxtimes P$ and $\min\{|V(H)|,|V(P)|\} \leq O(\tw(G))$. It is even open whether $\min\{|V(H)|,|V(P)|\} \leq f(\tw(G))$ for some function $f$, or whether $\min\{|V(H)|,|V(P)|\} \leq O(\sqrt{|V(G)|})$ (which would be implied since $\tw(G) \leq O(\sqrt{|V(G)|})$ for every planar graph $G$).

{\fontsize{10pt}{11pt}\selectfont
\bibliographystyle{DavidNatbibStyle}
\bibliography{DavidBibliography,additionalbiblio}}
\end{document}